\documentclass[a4paper,12pt]{amsart}
\usepackage{amsmath, amsfonts, amssymb, amsthm}
\usepackage{mathtools}
\usepackage{caption}
\usepackage{subcaption}

\begin{document}

\title[The multineighbor complex of a random graph]{The asymptotic topology of the multineighbor complex of a random graph}
\author{Wojciech Matysiak}

\author{Jan Spali\'nski}



\address{Faculty of Mathematics and Information Science,
              Warsaw University of Technology,
              Koszykowa 75,
              00-662 Warsaw,
              Poland}

\email{wojciech.matysiak@gmail.com}
\email{j.spalinski@mini.pw.edu.pl}

\newtheorem{lemma}{Lemma}
\newtheorem{prop}{Proposition}
\newtheorem{thm}{Theorem}
\newtheorem{cor}{Corollary}

\def\v#1{{\bf V}(#1)\ }
\def\e#1{{\bf E}(#1)\ }
\DeclarePairedDelimiter{\ceil}{\lceil}{\rceil}


\maketitle

\begin{abstract} We introduce the multineighbor complex of a graph, which is a simplicial complex
in which a simplex is a subset of the graph with a sufficient number of mutual neighbors. We investigate the asymptotic
homological properties of such complexes for the Erd\H os--Renyi random graphs and obtain a number of vanishing and nonvanishing
results. We use this construction to perform a topological data analysis classification of noisy synthetic point clouds obtaining
favorable accuracy as obtained by the standard methods. The case when there is a single neighbor has been studied earlier
by Mathew Kahle.
\end{abstract}

\section{Introduction}  The aim of this paper is to introduce a certain filtered collection of simplicial complexes which is associated with a (random) graph and to investigate its properties. The starting point of our research is the paper of M. Kahle \cite{K1}.

Given a graph $G$ and a positive integer $m$, a $q$-simplex of the multineighbor complex $N_m(G)$ is a collection of $q+1$ vertices in $G$ which has at least $m$ common neighbors. The case $m=1$ has been considered in \cite{K1}.

The motivation for this study comes from topological data analysis (TDA, see G. Carlsson \cite{C}, Carlsson and A. Zomorodian \cite{CZ} and H. Edelsbrunner, D. Letscher and Zomorodian \cite{ELZ}). Starting with a point cloud (formally, a finite subset of a  metric space) and a proximity parameter, a family of simplicial complexes is constructed.
This family is used to obtain homological invariants of the point cloud (persistent homology).
The most natural construction of such a complex is the Rips-Vietoris (a.k.a clique) complex.
A deficiency of this construction for data analysis is the size of the complexes involved (already the number of vertices is the number of elements of the point cloud).
Our constructions could be summarized as a parametrized version of the neighborhood complex, where
the extra parameter ($m$ - number of neighbors) reflects the density of the graph around a given simplex.

The paper is organized as follows. Section 2 presents the multineighbor complex and gives a few simple examples.
Section 3 (respectively, \S4) presents  a number of vanishing (respectively, nonvanishing) theorems for the reduced homology of multineighbor complexes of random graphs.
Section 5 describes an application of the multineighbor complex to topological data analysis. A point cloud determines a bigraded persistence diagram in a natural way (via the proximity parameter $r$ and the density parameter $m$). Using topological entropy as a classifier of a noisy point cloud in Euclidean space, our experiments show that the multineighbor complex gives greater accuracy of the random forest classifier than either the neighborhood complex, the Vietoris--Rips complex or the Alpha complex.

The collection of all possible modifications of the idea of a neighborhood complex of a graph is considered in the PhD. dissertation of C. Previte \cite{C}.


\section{The multineighbor complex of a graph}

In this section we introduce and study the multineighbor complex of a graph. By a graph $G$ we mean an undirected graph,
whose vertices will be denoted by $\v G$ and edges $\e G$.

In the definition below $m$ plays the role of a density parameter and $q$ is the dimension of the simplex.
Figure 1 below shows a graph in which each vertex has at least 3 neighbors.

\begin{figure}[h!]
\centering
\includegraphics[width=0.5\linewidth]{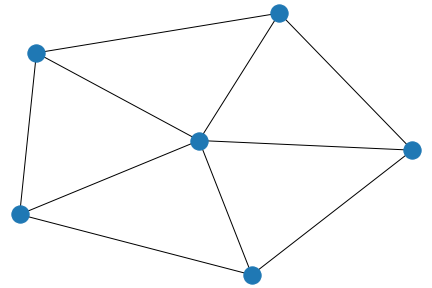}
\caption{A graph in which each vertex has at least 3 neighbors}
\end{figure}

{\bf Definition.}\label{multi_neighbor}
Let $G$ be a graph and let $m$ and $q$ be nonnegative integers.
{\sl The multineighbor complex} of the graph $G$, denoted $N_m(G)$, is obtained by declaring  a collection of $q+1$ vertices to be a $q$-simplex if these vertices have at least $m$ common neighbors in $G$.
\medskip

When $m=1$ we have the usual definition of the neighborhood complex studied in \cite{K1}, denoted  by $N(G)$. The simplest example is provided by the complete graph with $n$ vertices $K_n$, where the multineighbor complex is a skeleton of the standard simplex of dimension $n-1$.

Given a simplicial complex $L$, we let $L^{(i)}$ denote the $i$-dimensional skeleton of $L$, i.e. the set of simplices of $L$ of dimension at most $i$.

\begin{lemma}
Let $K_n$ be the complete graph on $n$ vertices. The multineighbor complex of $K_n$ is the $n-1-m$ dimensional skeleton of the standard simplex of dimension $n-1$ on the vertices of $K_n$
\[ N_m(K_n) =(\Delta^{n-1})^{(n-1-m)} \]
\end{lemma}

\begin{proof} We describe the first two cases explicitly. If $m=0$, every collection of $q+1$ vertices in $K_n$
spans a simplex, hence the multineighbor complex is just $\Delta^{n-1}$, this simplicial complex of all nonempty subsets of the  $n$-element set of vertices of $K_n$. Next, if $m=1$, a $q+1$-tuple will have a neighbor in $K_n$ exactly when $q+1<n$, or $q< n-1$. Hence for $m=1$ we obtain the $n-2$ skeleton of $\Delta^{n-1}$. In general, because the graph is complete and has $n$ vertices, a $(q+1)$-tuple will have a neighbor if
$q+1+m\le n$. The result follows.
\end{proof}

Figures 2  and 3 show the 1-skeletons of the multineighbor complexes of two simple graphs, where the number of neighbors is $m=1$ and $m=2$.

\begin{figure}[h!]
\centering
\includegraphics[width=1.1\linewidth]{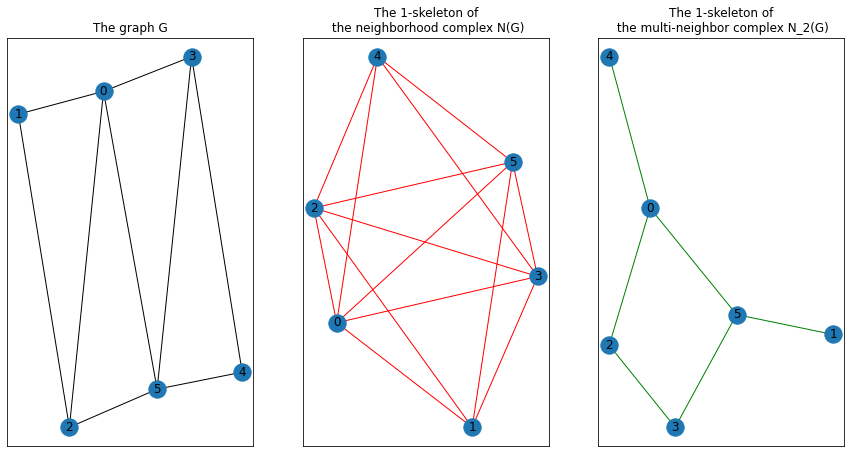}
\caption{A graph and the 1-skeletons of the multineighbor complexes with $m=1$ and $m=2$}
\end{figure}

\begin{figure}[h!]
\centering
\includegraphics[width=1.1\linewidth]{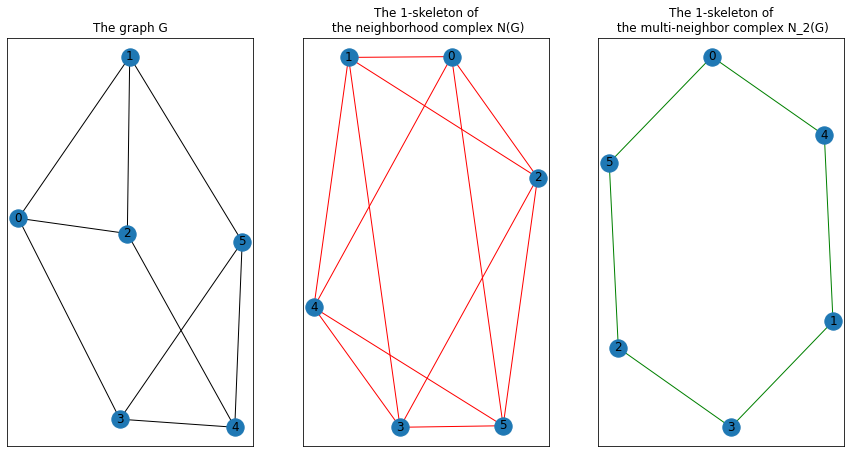}
\caption{A graph and the 1-skeletons of the multineighbor complexes with $m=1$ and $m=2$}
\end{figure}

\hfill

\eject

\section{Vanishing theorems for reduced homology}

We will need the following result of A. Bjorner (which is a consequence of D. Quillen's Order homotopy theorem).
Recall that given a partially ordered set $P$, the order complex $\Delta(P)$ is a simplicial complex whose vertices are the elements of $P$ and whose
$q$ simplices are the $q+1$-element chains of elements of $P$: $p_0<p_1<\dots<p_q$. Given a simplicial complex $L$, we let $P(L)$ be the poset of its
faces ordered by inclusion. One can check that $\Delta(P(L))$ is the barycentric subdivision of $L$, and hence the two simplicial complexes are homeomorphic.
The operations $P(\cdot)$ and $\Delta(\cdot)$ allow us to think of simplicial complexes and partially ordered sets as essentially equivalent as models of topological spaces. In particular, when we consider a topological notion for a partially ordered set $P$, this means that we are considering that notion for the underlying space of the simplicial complex $\Delta(P)$.

\medskip

\noindent{\bf Corrolary (A. Bjorner, page 26, Cor 11.8)}: {\it Let $P$ be a partially ordered set and $f:P\to P$ be an order-preserving map such that
\[f(x)\ge x\qquad\textrm{for all}\qquad x \in P\]
Then $f$ induces a homotopy equivalence between $P$ and $f(P)$.

If, moreover, $f^2(x)=f(x)$  for all $x\in P$, then $f(P)$ is a strong deformation retract of $P$.
}
\medskip

On a number of occasions we will use the following theorem, which is the multineighbor analogue of Kahle's
Lemma 2.3 in \cite{K1}.

\begin{thm}\label{main_theorem} Let $H$ be a graph, $a$ and $b$ positive integers, and $m$ an integer such that $1\le m\le a+b-2$. If $H$ does not contain a complete bipartite graph $K_{a,b}$, then $||N_m(H)||$ strong deformation retracts to a complex of dimension at most $a+b-m-2$.
\end{thm}

\begin{proof} We will prove the contrapositive. Hence assume that $||N_m(H)||$ does not deformation retract
to a complex of dimension at most $a+b-m-2$, and hence every deformation retract must have dimension at least $a+b-m-1$.

For a vertex $x\in \v H$, let $\Gamma(x)$ be the set of neighbors of $x$. If $X$ is a subset of $\v H$, we define
\[  \Gamma(X)  = \bigcap_{x\in X} \Gamma(x)  \]
It is easy to see that $\Gamma$ is an order reversing map from the set of all subsets of $\v H$ to itself.
It follows that $\Gamma^2$ is an order preserving map. Here is a simple example. Consider the graph $H$ displayed in Figure 4.

\begin{figure}[h!]
\centering
\includegraphics[width=0.6\linewidth]{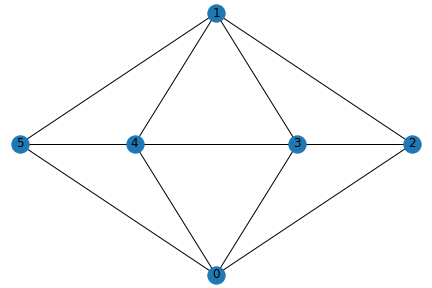}
\caption{The graph $H$ used in the description of the $\Gamma$ construction}
\end{figure}

We have $\Gamma(\{0\})=\{2,3,4,5\}$ and $\Gamma(\{2,3,4,5\})=\{0,1\}$.

Let $X$ be an arbitrary subset of $\v H$. Since every element of $X$ is connected to each element of $\Gamma(X)$, we conclude $X\subset \Gamma^2(X)$. Since $\Gamma$ reverses inclusions, we also have
$\Gamma^3(X)\subset \Gamma(X)$. Setting $X=\Gamma(Y)$ in the earlier formula, we obtain $\Gamma(Y)\subset \Gamma^3(Y)$.

We conclude $\Gamma^3(X)=\Gamma(X)$ for all $X \subset \v H$.

Let $P=P(N_m(H))$ denote the partially ordered set of simplices of $N_n(H)$, which in turn are certain subsets of $\v H$.

Let $v:P\to P$ be the map defined by the formula
\[   v(X) = \Gamma^2(X) \cap N_m(H)  \]
Below, we show that $v^2=v$.
We begin by showing that $v(X) \subset v^2(X)$.

Let $Y=v(X)$. We have
\[  Y \subset N_m(H) \qquad\textrm{and}\qquad Y\subset \Gamma^2(Y)  \]
Hence
\[  Y \subset N_m(H) \cap \Gamma^2(Y) = v(Y) \]
Recalling that $Y=v(X)$, we have
\[ v(X) \subset v(v(X)) = v^2(X) \]

Next, we prove the reverse inclusion, i.e. $v^2(X) \subset v(X)$.

To begin, note that $v$ is monotonic, that is
\[    X\subset Y \subset N_m(H) \implies v(X) \subset v(Y)            \]
This follows since we know that for $X\subset Y$, $\Gamma^2(X) \subset \Gamma^2(Y)$, and so
\[  v(X) =  \Gamma^2(X) \cap N_m(H) \subset \Gamma^2(Y) \cap N_m(H) = v(Y)\]

Hence we have,

\begin{align*}
 v^2(X) &= \Gamma^2\left(\Gamma^2(X)\cap N_m(H)\right) \cap N_m(H) \\
       &\subset \Gamma^2\left(\Gamma^2(X)\right) \cap N_m(H)  \\
       &=\Gamma^2(X)\cap N_m(H) = v(X)
\end{align*}

This finishes the proof that $v^2(X) = v(X)$.

By Corollary 11.8 in Bjorner \cite{Bj}, $v(P)$ is a strong deformation retract of $P$.

By the assumption made at the beginning of the proof, $v(P)$ must have dimension at least $a+b-m-1$.
Hence it must contain at least one simplex of that dimension, say of dimension $q$,
with $q\ge a+b-m-1$.

Such a simplex in $v(P)$ is a strictly increasing sequence of simplices in $N_m(H)$ of the form:

\[  X_1 \subsetneq X_2 \subsetneq \dots \subsetneq X_{q+1}  \qquad (*) \]

Since $X_1$ is nonempty and the inclusions are strict, $X_i$ contains at least $i$ vertices of $1\le i\le q+1$.

Letting $Y_i= \Gamma(X_i)$, we have

\[  Y_1 \supsetneq Y_2 \supsetneq \dots \supsetneq Y_{q+1}  \qquad (**) \]

Since $\Gamma^3 = \Gamma$, the inclusions immediately above are strict. Here we need to recall that
$(*)$ is a simplex in $v(P)$, hence each $X_i$ is of the form $X_i= \Gamma^2(U_i)$ for some $U_i$ in $P$.

Since $X_{q+1}$ is a simplex in $N_m(H)$, there are at least $m$  vertices in $H$ outside of $X_{q+1}$ connected to each vertex of $X_{q+1}$.

Since the sequence $(**)$ is  strict, and $Y_{q+1}$ contains at least $m$ elements, it follows that
$Y_{q+1-i}$ contains at least $m+i$ vertices.

Now there are two cases to consider:

\begin{enumerate}
\item $m\le \max\{a,b\}$,
\item $m > \max\{a,b\}$.
\end{enumerate}

{\bf Case 1.} Without loss of generality, assume $m\le b$.

Setting $i=b-m$, we see that $Y_{q+1-i}$ has at least $m+(b-m) = b$ vertices.

Recall from our choice of $q$, that we must have
\[   q+1-i \ge  a + b - m -1 + 1 - (b-m) = a           \]

Hence $X_{q+1-i}$ has at least $a$ vertices which along with the vertices of $Y_{q+1-i}$ span
a complete bipartite graph in $H$. This contradicts the assumptions made in the statement of the theorem.

\medskip
{\bf Case 2.} We now have $m>a$ and $m>b$, and every vertex of $X_{q+1}$ is connected to every vertex of  $Y_{q+1}$,
and each of these has at least $m$ vertices (as they lie in $N_m(H)$). Hence the vertices of $X_{q+1}$ and   $Y_{q+1}$
span a complete bipartite subgraph of $H$, with the partite subsets of cardinality greater than $a$ and $b$. This again
contradicts the assumptions made in the statement of the theorem.

\end{proof}

Let $G(n,p)$ be a random graph with $n$ vertices and each edge appearing independently of others with probability $p$ (i.e.
this is an Erd\H os--R\'enyi random graph \cite{ER}). 
Recall that the statement that a sequence $a_n$ is ${\bf o}(1)$ means that it has limit zero as $n$ approaches infinity. We will say that an object $X$ asymptotically almost surely has property $P$, abberivated a.a.s., if $\lim_{n\to\infty} \Bbb P(X\in P)=1$.

\begin{lemma} Let $m$ be a fixed positive integer, and  $p=p(n)$ be a monotone sequence with values in the interval $[0,1]$. If
\[  {n\choose i} \left(1-p^{mi}\right)^{{n-i}\choose m}  = {\bf o}(1) \]
then $N_m(G(n,p))$ is a.a.s. $(i-2)$-connected.
\end{lemma}

 \begin{proof} A simplicial complex  is $i$-neighborly if every set of $i$ vertices spans a face; this property implies that the simplex is $(i-2)$-connected. The probability that a given set of $i$ vertices in $G(n,p)$ has a given set of $m$ vertices as neighbors has probability $p^{mi}$. Hence the probability of the opposite is $1 - p^{mi}$. Hence the probability that a given set of $i$ vertices does not have a set of $m$ neighbors in
$G(n,p)$ is equal to $\left(1-p^{mi}\right)^{{n-i}\choose m}$. Hence, finally, the probability that no set of $i$ vertices has $m$ neighbors is bounded from above by
\[ {n\choose i} \left(1-p^{mi}\right)^{{n-i}\choose m}  \]
\end{proof}

\begin{prop} If $p=\frac{1}{2}$ and $\varepsilon>0$, then a.a.s. $H_{l-2}(N_mG(n,p))=0$ for
$l\le \left(1-\frac{\varepsilon}{m}\right) \log_2 n$.
\end{prop}

\begin{proof} In what follows, we will use the following well-known inequalities:

\begin{align*}
  1-x &\le \exp(-x) , \qquad x\in \mathbb R \\
 \frac{n^k}{k^k} &\le {n \choose k}  ,\qquad  1 \le k \le n
\end{align*}

By the above lemma, it is enough to show that
\[  {n\choose l} \left(1-p^{ml}\right)^{{n-l}\choose m}  = {\bf o}(1) \]

We have
\begin{align*}
{n\choose l} \left(1-p^{ml}\right)^{{n-l}\choose m}
        &\le n^l \exp\left(-\left(\frac{1}{2}\right)^{ml}{{n-l}\choose m}\right)\\
        &\le n^l \exp\left(-\frac{1}{2^{(m-\varepsilon)\log_2 n}}{{n-l}\choose m}\right)\\
        &\le n^l \exp\left(-\frac{1}{n^{(m-\varepsilon)}}{{n-l}\choose m}\right)\\
        &\le n^l \exp\left(-{n^{(-m+\varepsilon)}}\cdot {\frac{(n-l)^m}{m^m}}\right)\\
        &= \exp\left(l\ln(n) -\left(\frac{1}{m^m}\right){n^{(-m+\varepsilon)}}{{(n-l)^m}}\right)\\
        &= \exp\left(l\ln(n) -\left(\frac{1}{m^m}\right)
                 \left[{n^{\varepsilon}} - m\,l\, n^{\varepsilon-1} + \dots\right]\right)\\
        &={\bf o}(1)
\end{align*}
The last equality following from the fact that $\lim_{x\to \infty} \frac{x^{\varepsilon}}{\ln x} =\infty$.

\end{proof}

\begin{prop} If $p=\frac{1}{2}$ and $\varepsilon>0$, then a.a.s. $H_{l-m+1}(N_m G(n,p))=0$ for
$l\ge \left(4+\varepsilon\right) \log_2 n$.
\end{prop}

\begin{proof} Set $j=k=\ceil[\bigg]{\displaystyle\frac{l}{2}}$. The probability of finding a $(j,k)$-bipartite subgraph in $G(n,p)$ is bounded above by
\[ {n \choose j} {n \choose k} \left(\frac{1}{2}\right)^j \left(\frac{1}{2}\right)^k \]

We will show that the above expression is ${\bf o}(1)$ for the given choice of $l$ (and hence $j$ and $k$).

\begin{align*}
	{n \choose j} {n \choose k} \left(\frac{1}{2}\right)^j \left(\frac{1}{2}\right)^k
	                  &\le n^j n^k \frac{1}{2^{jk}} \\
	                  &\le n^{j+k} \frac{1}{2^{jk}} \\
                      &\le n^{l+2} \frac{1}{2^{jk}} \\
                      &\le n^{l+2} \frac{1}{2^{\frac{l^2}{4}}} \\
                      &\le n^{l+2} \frac{1}{(n^{4+\varepsilon})^{\frac{l}{4}}} \\
                      &\le n^{l+2} \frac{1}{(n^{l+\varepsilon l/4})} \\
                      &\le n^{2- \frac{\varepsilon}{4} l} = {\bf o}(1)
\end{align*}

By Theorem \ref{main_theorem}, $N_m(G(n,p))$ deformation retracts to a complex of dimension at most
\[   j+k - m - 2 \le \frac{l}{2} + 1 + \frac{l}{2} + 1 - m - 2 = l - m   \]

Hence the homology of $N_m(G(n,p))$ vanishes in dimensions $l-m+1$ and above.

\end{proof}

\begin{prop} Let $l$ be a nonnegative integer and $p=n^{\alpha}$, where $\alpha\in \left[\displaystyle\frac{-1}{l+2},0\right]$. Then a.a.s. $H_l(N_m(G(n,p)))=0$.
\end{prop}

\begin{proof} The proof consists of finding an upper bound for the probability of finding an $l+2$ simplex
in $N_m(G(n,p))$. We have

\begin{multline*}
	{n \choose l+2}  \left(1 - p^{m(l+2)}\right)^{n-(l+2) \choose m}
	            \le n^{l+2} \exp\left(-p^{m(l+2)} {n-(l+2) \choose m} \right) \\
	            \le n^{l+2} \exp\left(-n^{m\alpha (l+2)} {n-(l+2) \choose m} \right) \\
	            \le n^{l+2} \exp\left(-n^{m\alpha (l+2)} \,  \frac{(n-(l+2))^m}{m^m} \right) \\
	            = n^{l+2} \exp\left(-n^{m\alpha (l+2)} \, (n-(l+2))^m \frac{1}{m^m} \right) \\
	            = n^{l+2} \exp\left(-n^{m\alpha (l+2)} \, \left(n^m-m(l+2)n^{m-1} +\dots \right)\frac{1}{m^m} \right) \\
	            = n^{l+2} \exp\left(\left(-n^{m(\alpha (l+2)+1)} + m(l+2)n^{m(\alpha (l+2))} +\dots \right)\frac{1}{m^m} \right) ={\bf o}(1)\\
\end{multline*}

By the Lemma above,  $N_m(G(n,p))$ is a.a.s. $l$-connected, hence its reduced homology in dimension $l$ vanishes.

\end{proof}

\begin{prop} Let $l$ and $m$ be a positive integers and $p=n^{\alpha}$, where

\begin{enumerate}
\item[(a)] for $l$ even $\alpha\in \left[-2,\displaystyle\frac{-4}{l+2}\right]$ ,
\item[(b)] for $l$ odd $\alpha\in \left[-2,\displaystyle\frac{-4(l+2)}{(l+1)(l+3)}\right]$ ,
\end{enumerate}
Then a.a.s. $\tilde H_l (N_m(G(n,p))) = 0$.
\end{prop}

\begin{proof}  We apply the argument given in the proof of Corollary 2.5 in Kahle \cite{K1} on page 384, and apply our Theorem \ref{main_theorem} in order to obtain the result.
\end{proof}

\section{Non-vanishing theorems for reduced homology}

We now turn to establishing that certain homology groups of the multineighbor complex do not vanish in the asymptotic sense for certain random graphs $G(n,p)$ under the appropriate conditions on $p$. It is usual to refer to a subgraph which is a complete graph as a clique.

{\bf Definition.} Let $k$ and $m$ be positive integers. The $X_{k,m}$ graph is the graph with the
vertex set
\[   \{u_1,\dots,u_k\} \cup \{v_{1,l},\dots,v_{k,l}\}_{l=1,\dots,m}       \]
such that  $\{u_1,\dots,u_k\}$ spans a clique and $u_i$ is adjacent to $v_{j,l}$ for $i\neq j$ and $l=1,\dots,m$.
\medskip

The $X_{3,2}$ graph is displayed in Figure 5.

\begin{figure}[h!]
\centering
\includegraphics[width=0.8\linewidth]{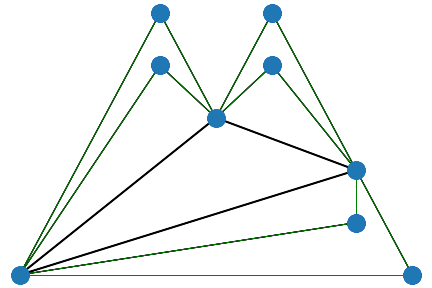}
\caption{The $X_{3,2}$ graph}
\end{figure}

\begin{prop} Let $m$ and $k$ be integers with $m\ge 1$ and $k\ge 3$. If the graph $H$ has a maximal clique which cannot be extended to an $X_{k,m}$ subgraph, then $||N_m(H)||$ retracts to the sphere
$\mathbb S^{k-2}$
\end{prop}

\begin{proof} The proof consists of the observation that the argument given in \cite{K1} may be repeated here once one replaces the graph $X_n$ considered there by the graph $X_{k,m}$.
\end{proof}

Of course the applications of this result will require a different argument, as one will need to determine the absence of a different type of subgraph.

\begin{prop}\label{main_prop} Let $p=\frac{1}{2}$, $\varepsilon>0$,  and $k$ be an integer satisfying the inequalities
\[     \left( \frac{2m+2}{2m+1} +\varepsilon\right) \log_2 n < k < (2-\varepsilon) \log_2 n             \]
Then a.a.s.
\[    \tilde H_k(N_m(G(n,p))) \neq 0       \]
\end{prop}

\begin{proof}  It follows from B. Bollobas \cite{Bl}, that $G(n,p)$ contains maximal cliques of every order $k$ in the interval
\[   (1+\varepsilon) \log_2 n < k < (2-\varepsilon) \log_2 n    \]
It remains to show that under our assumptions there are a.a.s. no $X_{k,m}$ subgraphs.

The probability of finding an $X_{k,m}$ subgraph is bounded from above by the following:

\begin{align*}
	[(m+1)k]! &{n \choose (m+1)k}  \left(\frac{1}{2}\right)^{\frac{(2m+1)k(k-1)}{2}} \\
	            &\le n^{(m+1)k} \left(\frac{1}{2}\right)^{\frac{(2m+1)k(k-1)}{2}}\\
	            & = n^{(m+1)k} 2^{-\frac{(2m+1)k(k-1)}{2}}\\
	            & = n^{(m+1)k} n^{\log_n 2[-\frac{(2m+1)k(k-1)}{2}]}\\
	            & = n^{(m+1)k} n^{[-\frac{(2m+1)k(k-1)}{2\log_2 n}]}\\
	            & = n^{[\frac{2(m+1)k \log_2 n - (2m+1)k(k-1)}{2\log_2 n}]}\\
	            & = n^{k[\frac{2(m+1) \log_2 n - (2m+1)(k-1)}{2\log_2 n}]}\\
& \le n^{k[\frac{2(m+1) \log_2 n - (2m+1)\left(\left[\frac{2m+2}{2m+1}+\varepsilon\right]\log_2 n-1\right)}{2\log_2 n}]}\\
& = n^{k[\frac{- (2m+1)\varepsilon\log_2 n +(2m+1)}{2\log_2 n}]}\\
\end{align*}

Hence the exponent of $n$ is of the order of
\[       \frac{- (2m+1)k\varepsilon  }{2}      \]
and hence the bound is {\bf o}(1) as $n$ goes to infinity.

\end{proof}

Recall that the density of a graph with $v$ vertices and $e$ edges is defined as $\lambda = \displaystyle\frac{e}{v}$. A graph is balanced if its density is greater or equal to that of each
of its subgraphs.

\begin{lemma} The graph $X_{k,m}$ is balanced and has density given by
\[            \lambda = \frac{(2m+1)(k-1)}{2(m+1)}              \]
\end{lemma}

\begin{proof} Clearly, we have
\[   \v {X_{k,m}} = (m+1) k                      \]
\[   \e {X_{k,m}} = \frac{k(k-1)}{2} + m k (k-1) = \frac{(2m+1)k(k-1)}{2}             \]
Observe that the graph $X_{k,m}$ is obtained from the balanced complete graph on $k$ vertices by successively adding vertices of order $k-1$. The densities of the graphs obtained in this way are given
by the sequence
\[   \lambda_n = \frac{\frac{k(k-1)}{2} + n(k-1)}{k+n}           \]
It is a routine calculation to show that this is a strictly increasing sequence.
One completes the argument by observing that if one removes a vertex of the clique, then the density
of the remaining graph will be smaller than if one removes one of the vertices from outside the clique.

\end{proof}

From  Theorem 4.4.2  in N. Alon and J. Spencer \cite{NAJS}, we know that for a balanced graph $H$ with density $\lambda$,  the sharp threshold for $G(n,p)$ containing $H$ as a subgraph is given by
\[p=n^{ -1/\lambda}\]
That is, if $p=n^{\alpha}$ with $\alpha> \displaystyle -\frac{1}{\lambda}$, then $G(n,p)$ a.a.s. contains $H$
as a subgraph and if $\alpha <\displaystyle -\frac{1}{\lambda}$, then $G(n,p)$ a.a.s. does not contain $H$ as a subgraph.

\begin{prop} If $p=n^{\alpha}$ with
\[   \frac{-2}{k+1} <  \alpha < \frac{-2(m+1)}{(2m+1)(k+1)}             \]
then a.a.s. $\tilde H_k(N_m(G(n,p))) \neq 0 $.
\end{prop}

 \begin{proof} Since both $K_{k+2}$ and $X_{k+2,m}$ are  balanced, it follows from the remarks before the proposition that $G(n,p)$ a.a.s. contains cliques of order $k+2$ but no $X_{k,m}$ subgraphs. In order to apply Proposition 5 above, it remains to show that the probability of extending a given clique of order $k+2$ to one of order $k+3$ is {\bf o(1)}. In fact, this probability is bounded from above by
\[   n p^{k+2} = n \cdot n^{\alpha(k+2)} \le n \cdot n^{  \frac{-2(m+1)(k+2)}{(2m+1)(k+1)} }     =
                   n^{  \frac{-k - 2m -3}{(2m+1)(k+1)} } = {\bf o}(1)        \]

\end{proof}

\section{Applications to topological data analysis}

Given a finite metric space $\mathbb X$ (e.g. a point cloud in some Euclidean space), we can
construct a bi-filtered complex as follows. Let $r$ be a positive real number and $m$
a positive integer. The metric space $\mathbb X$ determines a filtered family of graphs $G(\mathbb  X,r)$,
where the point of $\mathbb X$ are the vertices, and there is an edge between two vertices if
the distance between them is less than $r$. This is simply the one skeleton of the Rips complex.
 If $r_1<r_2$, there is a natural map
\[  G(\mathbb X, r_1) \to G(\mathbb X, r_2)  \]
For every positive integer $m$ we can consider the multineighbor complex $N_m(G(\mathbb  X,r))$, which for fixed $r$ is
a decreasing filtration as a function of $m$.

There are now three types of persistence one can consider:
\begin{itemize}
\item with respect to $r$ keeping $m$ fixed;
\item with respect to $m$ keeping $r$ fixed:
\item with respect to both variables using multidimensional persistence of G. Carlsson and A. Zomorodian \cite{CZ2}.
\end{itemize}

Bellow we explore the first possibility above. The procedure described below is the adaptation to our setting of
the Giotto TDA tutorial {\it Topological feature extraction using Vietoris-Rips Persistence and Persistence Entropy}, \cite{Giotto}.

The main difference is that  we use the multineighbor complex rather than the Vietoris-Rips complex in order to compute
the persistent homology of the point cloud. For the very noisy examples we consider we obtain much better predictions
with the Random Forest classifier when using the multineighbor complex than when using the Vietoris-Rips complex.
To obtain calculate the relevant persistent homology of the multineighbor complex we use the simplex stream of the Gudhi TDA library \cite{Gudhi}.

Consider the following six subsets of $\mathbb R^3$: a circle, a sphere, a wedge of two spheres of different radii, a torus, and two types of pretzel.

\begin{figure}[h!]
    \centering
    \begin{subfigure}{.32\textwidth}
        \includegraphics[width=1.3\linewidth]{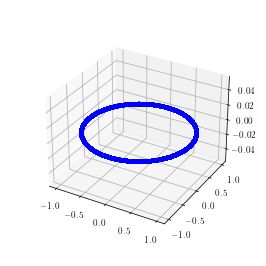}
        \caption{A circle}
    \end{subfigure}
    \begin{subfigure}{.32\textwidth}
        \includegraphics[width=1.3\linewidth]{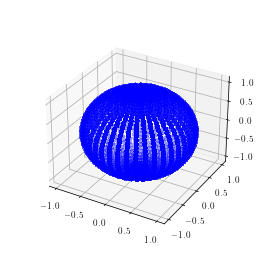}
        \caption{A sphere}
    \end{subfigure}
    \begin{subfigure}{.32\textwidth}
        \includegraphics[width=1.3\linewidth]{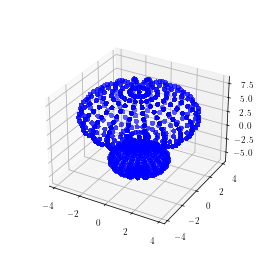}
        \caption{A double sphere}
    \end{subfigure}
    \begin{subfigure}{.32\textwidth}
        \includegraphics[width=1.3\linewidth]{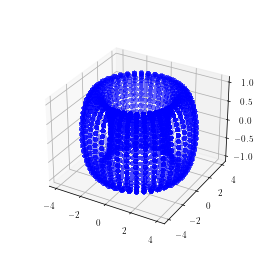}
        \caption{A torus}
    \end{subfigure}
    \begin{subfigure}{.32\textwidth}
        \includegraphics[width=1.3\linewidth]{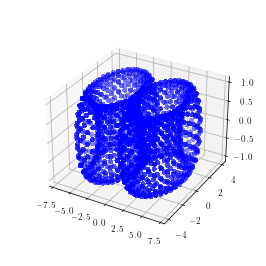}
        \caption{A pretzel (first type)}
    \end{subfigure}
    \begin{subfigure}{.32\textwidth}
        \includegraphics[width=1.3\linewidth]{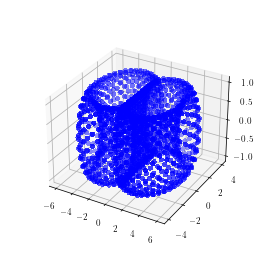}
        \caption{A pretzel (second type)}
    \end{subfigure}
    \caption{The point clouds without noise}
\end{figure}

\begin{figure}[h!]
    \centering
    \begin{subfigure}{.32\textwidth}
        \includegraphics[width=1.3\linewidth]{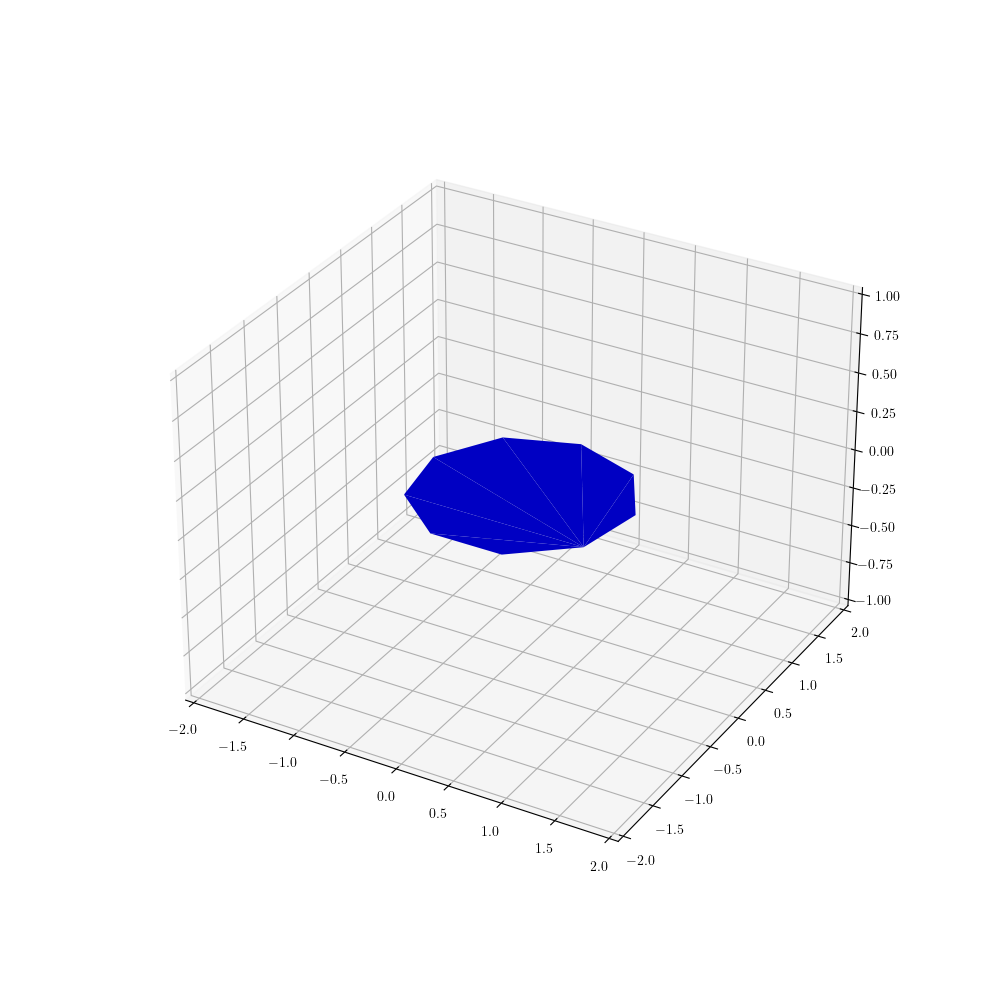}
        \caption{}
    \end{subfigure}
    \begin{subfigure}{.32\textwidth}
        \includegraphics[width=1.3\linewidth]{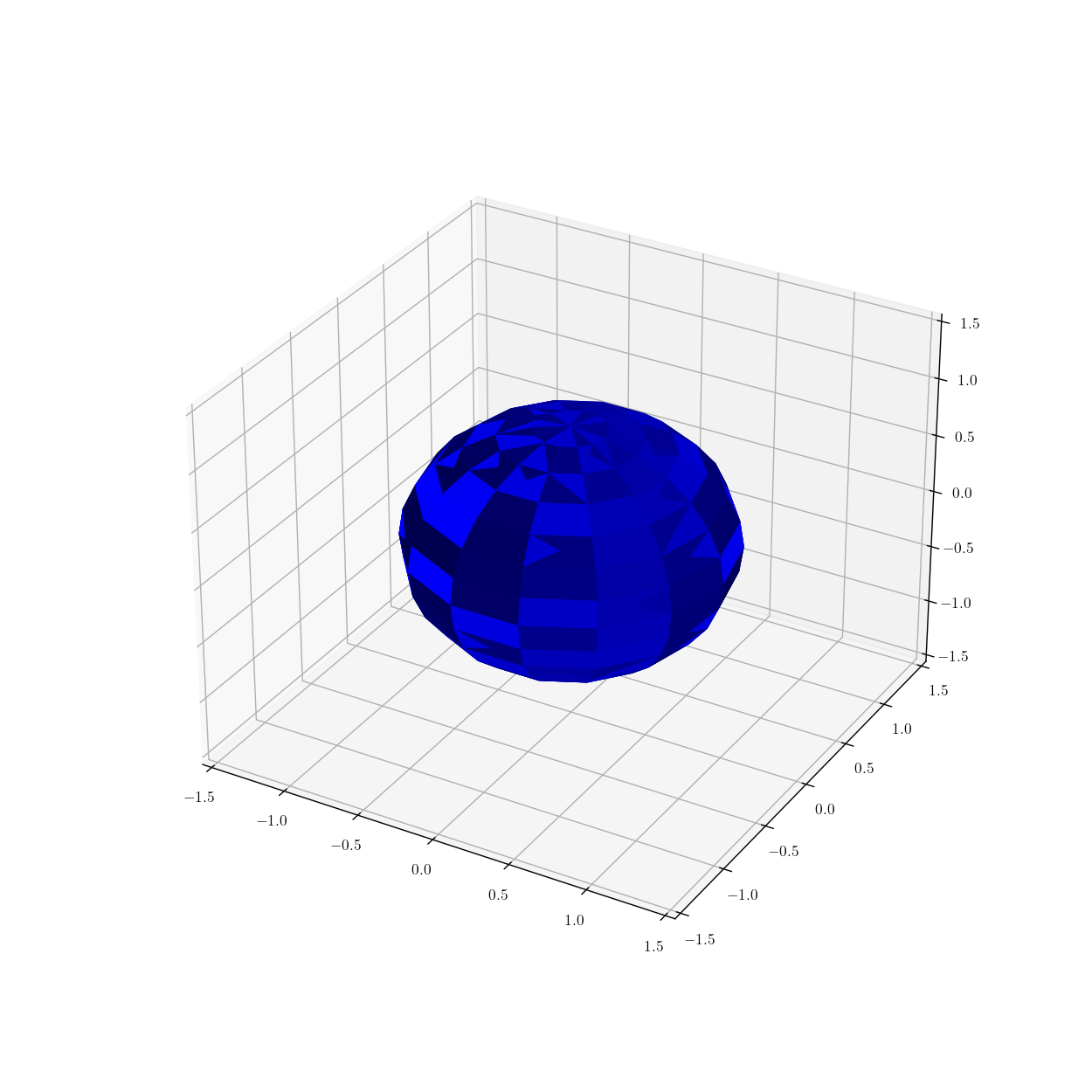}
        \caption{}
    \end{subfigure}
    \begin{subfigure}{.32\textwidth}
        \includegraphics[width=1.3\linewidth]{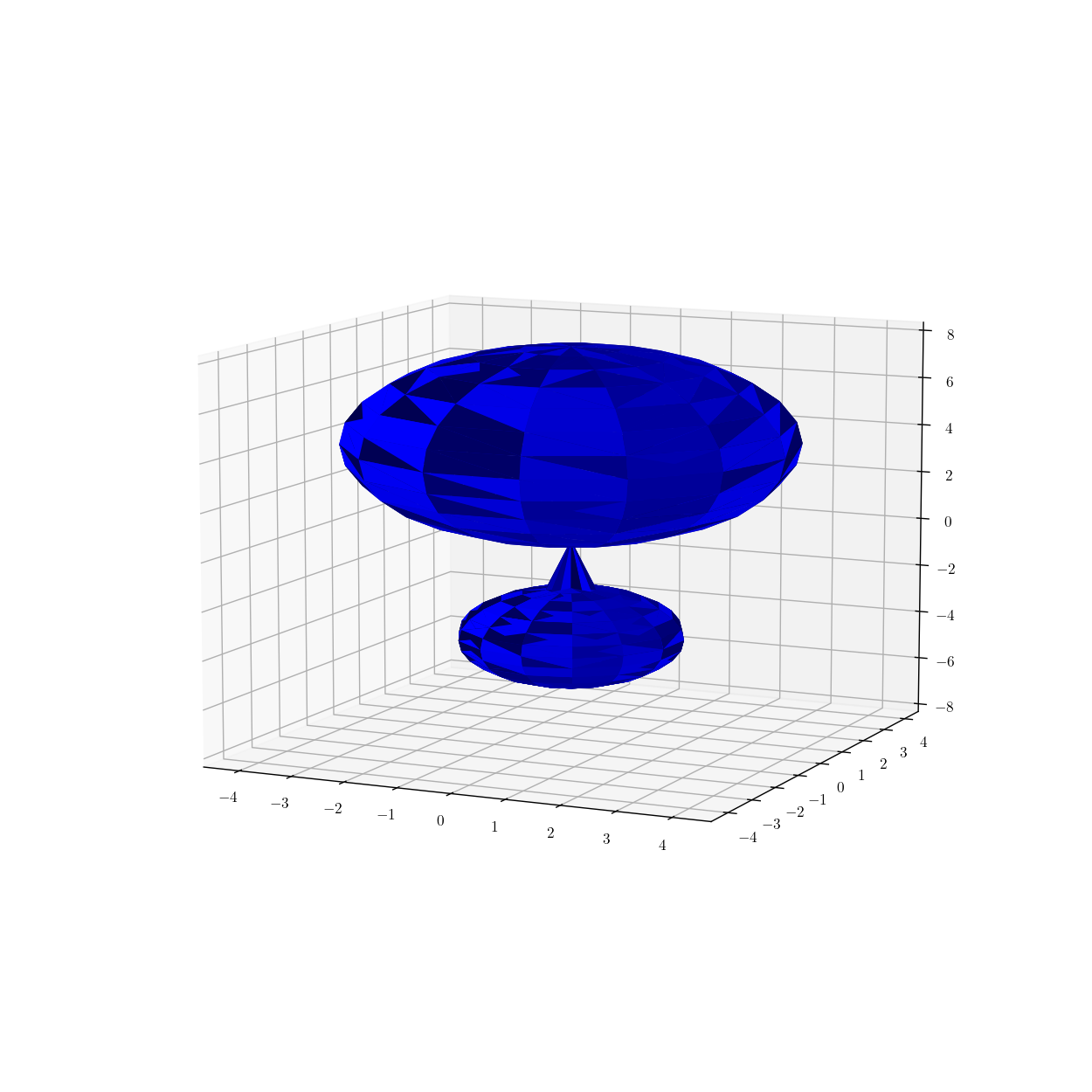}
        \caption{}
    \end{subfigure}
    \begin{subfigure}{.32\textwidth}
        \includegraphics[width=1.3\linewidth]{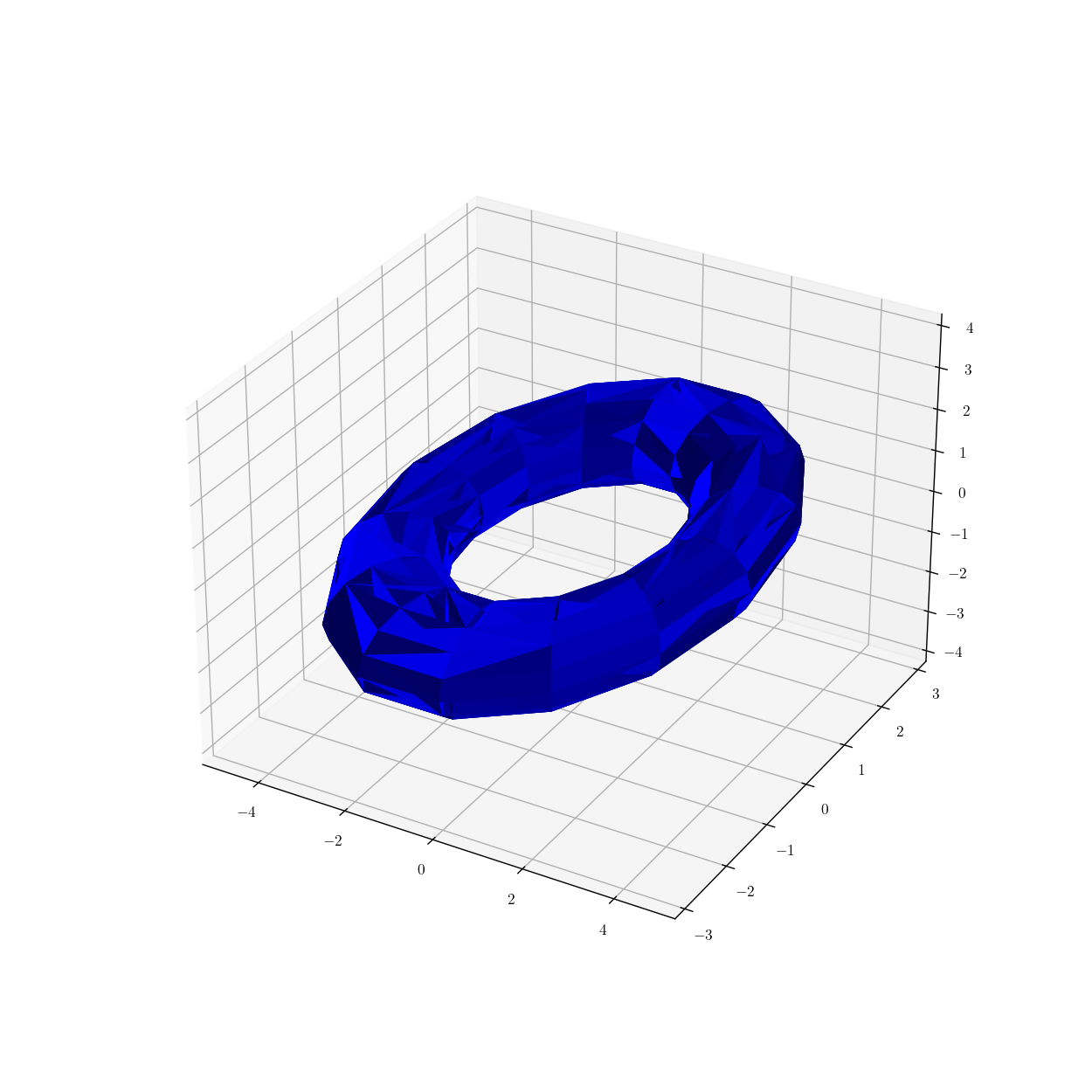}
        \caption{}
    \end{subfigure}
    \begin{subfigure}{.32\textwidth}
        \includegraphics[width=1.3\linewidth]{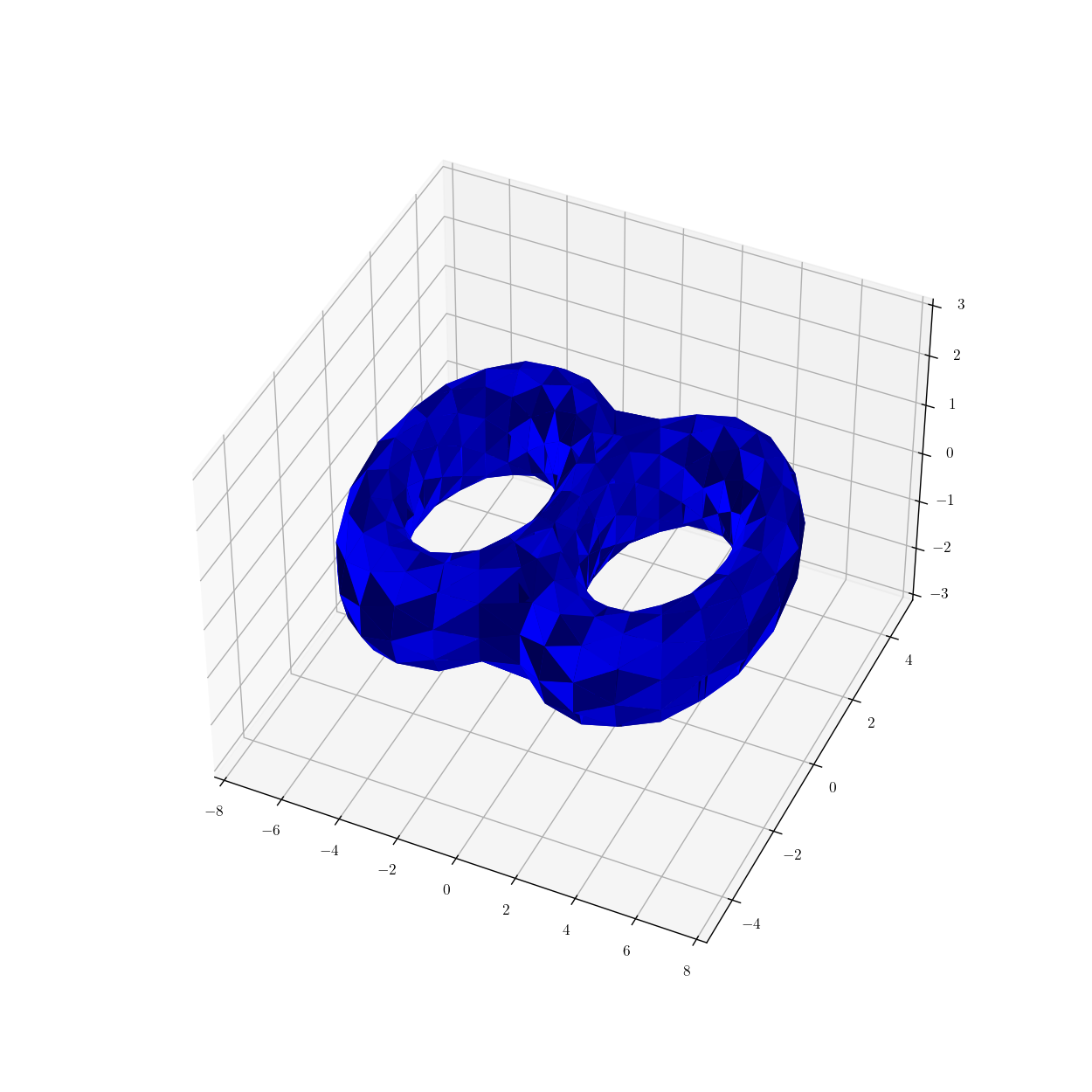}
        \caption{}
    \end{subfigure}
    \begin{subfigure}{.32\textwidth}
        \includegraphics[width=1.3\linewidth]{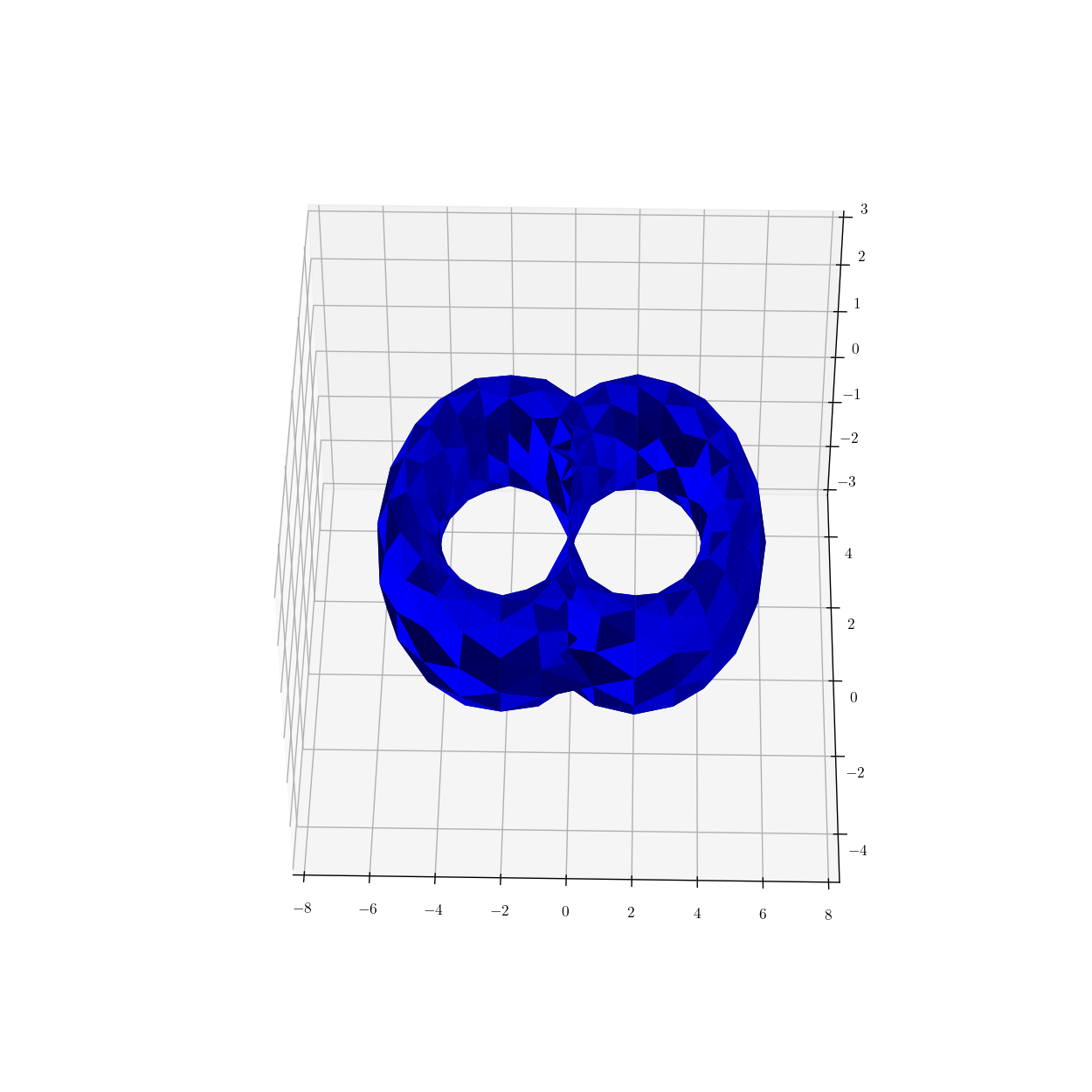}
        \caption{}
    \end{subfigure}
    \caption{Alpha Complex respresentations of the point clouds without noise}
\end{figure}

Except for the circle, each one is the image of a rectangular array
of points in $\mathbb R^2$ under a continuous map. The plots of these
subsets of $\mathbb R^3$ are displayed in Figure 6.

Now we consider a collection of 8 noisy samplings of each of these shapes (with noise of magnitude $1$
added independently to each coordinate; the diameter of each point cloud does not exceed 6). The case of the circle
is displayed in Figure 7.

\begin{figure}[h!]
\centering
\includegraphics[width=1.1\linewidth]{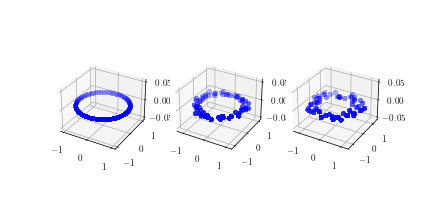}
\caption{Discrete samplings of the circle with various amounts of noise}
\end{figure}

In this way we obtain a collection of 48 point clouds. With each point cloud we associate a filtered
graph as follows: the vertices are the points of the point cloud. The edges are added as in the Rips complex
construction. Next, for values of m in the set $\{1,2,3,4,5 \}$ we construct the (filtered) multineighbor
complex and compute its persistent homology (using the Gudhi TDA package).  A discrete sampling of a torus with noise
and its persistence diagram is displayed in Figure 8.

\begin{figure}[h!]
\centering
\includegraphics[width=1\linewidth]{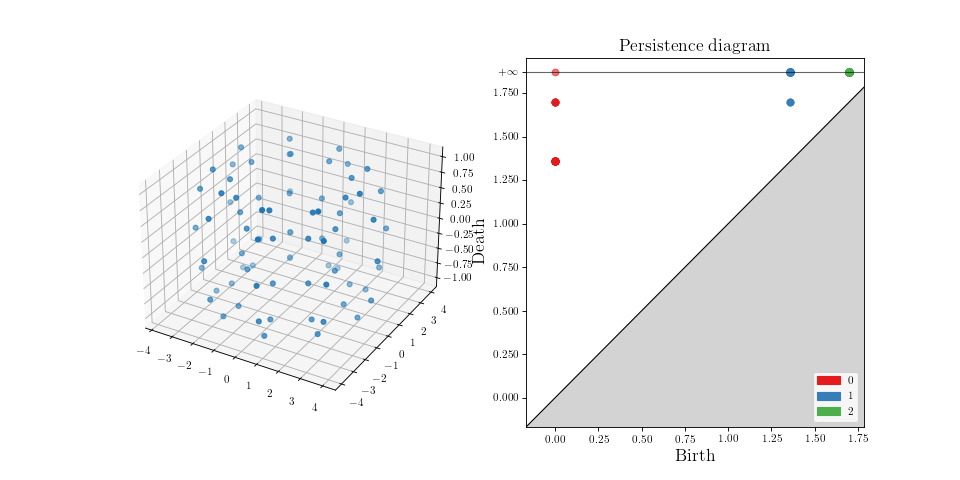}
\caption{A sampling of a torus with noise and its persistence diagram}
\end{figure}

In each homology dimension, we compute the entropy of the persistence diagram. This is defined as follows. Suppose that the
off-diagonal points in the persistence diagram of a given homology dimension form the set:
\[   \{ (b_i,d_i) \mid i=1, \dots,n  \} \]
i.e. there are $n$ persistent homology classes with birth times $b_i$ and death times $d_i$,  and we only consider those for which $b_i< d_i$.
With each homology class with associate the number
\[   p_i = \frac{d_i - b_i}{\sum_{i=1}^n (d_i - b_i)}   \]
These numbers are positive and add up to 1. We define the entropy by the familiar formula from probability and statistics:
\[     E = - \sum_{i=1}^n p_i \ln  p_i      \]

If a point cloud is altered in such a way that a single persistence class will move towards and merge with the diagonal, the entropy will
change in a continuous fashion (as follows from $\lim_{x\to 0+} x\ln x = 0$).

The vector of entropies in the homology dimensions $0$, $1$ and $2$
is used as the classifier of the point cloud. These are entered into the random forest classifier. The average classification
accuracy of twenty repetitions is entered into the table below. The first column gives the number of points in the point clouds considered.
The next five give the number of neighbors in the multineighbor complex considered (from 1 to 5). The final two columns show the results for the Vietoris--Rips complex
and the Alpha Complex, respectively.

\bigskip
\bigskip

\begin{tabular}{|c |c |c| c| c| c| c| c|c|}
\hline
Number of     & m=1 &  m=2  &  m=3 &   m=4 &  m=5 &  Vietoris-- &    Alpha      \\
points           &         &           &           &           &          &   Rips     &    complex \\
\hline
&  &  &  &  &  & &\\
36 & 0.87 & 0.91 & 0.88 & 0.93 & 0.89 & 0.69 & 0.74 \\
&  &  &  &  &  & &\\
\hline
&  &  &  &  &  & &\\
64 & 0.87 & 0.89 & 0.90 & 0.88 & 0.93 & 0.52 & 0.75\\
&  &  &  &  &  & &\\
\hline
\end{tabular}

\section{Conclusions}

We have established that the multineighbor complex of a (random) graph is a useful extension of the
neighborhood complex. It has a number of homological properties which are similar to the single-neighbor
case, however some of these properties  require substantially different proofs (especially the properties described in Theorem \ref{main_theorem} and Proposition \ref{main_prop}).
In Section 6 we apply the multineighbor complex to the classification of noisy point clouds in $3$-dimensional Euclidean space.
We conclude that when using topological entropy of a complex associated to our point clouds as a classifier, the multineighbor
complex has superior discriminating power as compared to either the neighborhood complex, the Vietoris--Rips complex or the Alpha Complex.


\end{document}